\documentclass[11pt]{article}
\usepackage{amsmath,amssymb,theorem}

\newcommand{\tmtextit}[1]{{\itshape{#1}}}

\newenvironment{proof}{\noindent\textbf{Proof\ }}{\hspace*{\fill}$\Box$\medskip}

\newtheorem{theorem}{Theorem}[section]
\newtheorem{proposition}[theorem]{Proposition}
\newtheorem{corollary}[theorem]{Corollary}
\newtheorem{definition}[theorem]{Definition}
\newtheorem{lemma}[theorem]{Lemma}
\newtheorem{remark}[theorem]{Remark}

\def \A {{\mathcal {A}}}
\def \N {{\mathcal {N}}}
\def \d {{\mathrm {d}}}

\begin{document}

{\bigskip}

\begin{center}
{\large Nijenhuis structures 
on Courant algebroids}\\

{\bigskip}

Yvette Kosmann-Schwarzbach\\

{\bigskip}

{\small Centre de Math\'ematiques Laurent Schwartz\\
\'Ecole Polytechnique\\ 
91128 Palaiseau, France\\

\texttt{yks@math.polytechnique.fr}}
\end{center}

{\bigskip}

\begin{abstract}
We study Nijenhuis structures on
Courant algebroids in terms of the canonical Poisson bracket
on their symplectic realizations.
We prove that the Nijenhuis torsion of a skew-symmetric 
endomorphism $\N$ of a Courant algebroid is skew-symmetric if  
$\N^2$ is proportional to the identity, and only in this case when the
Courant algebroid is irreducible.
We derive a necessary and sufficient condition for a skew-symmetric 
endomorphism to give rise to a deformed Courant structure.
In the case of the double of a Lie bialgebroid $(A,A^*)$, given an
endomorphism $N$ of $A$ that defines a skew-symmetric endomorphism
$\N$ of the double of $A$, we prove that the torsion of $\N$ is 
the sum of the torsion of $N$ and that of the
transpose of $N$. 
\end{abstract}

\medskip

\section*{Introduction}
The aim of this paper is to study the infinitesimal deformations of Courant
algebroids. We shall consider in detail 
the case of the double of a Lie bialgebroid, in particular the 
case of the double of a trivial Lie bialgebroid, such as the
generalized tangent bundle of a manifold.

Nijenhuis operators for algebras other than Lie
algebras were first considered
by Fuchssteiner \cite{F}, while their role 
in the study of contractions of Lie algebras was
first discussed by Bedjaoui-Tebbal \cite{BT}.
The theory of Nijenhuis operators in the case of 
general algebraic structures was developed in \cite{CGM0} by Cari\~ nena,
Grabowski and Marmo, who identified the role they play in the
theory of contractions and deformations 
of both Lie algebras and Leibniz (Loday) algebras.
The case of Courant algebroids \cite{LWX}, which are important
examples of Leibniz algebroids \cite{Poncin}  and whose spaces of sections are 
therefore Leibniz algebras \cite{L}, 
was considered by them in \cite{CGM}, and then by Clemente-Gallardo and
Nunes da Costa in \cite{nunes}, in both papers with applications to the
deformation of Dirac structures. More recently, in  \cite{G},  
Grabowski related the results obtained in \cite{CGM}
to Roytenberg's graded supermanifold approach to Courant algebroids
\cite{R}, while Keller and Waldmann in \cite{KW}
proceed by an alternative approach, using the Rothstein algebra in
their study of the deformations of Courant algebroids. 
Nijenhuis tensors on the double of a Lie bialgebroid were also
considered by Vaisman, in the case of the generalized
tangent bundle
$TM \oplus T^*M$ of a manifold, in his study of the reduction of
generalized complex manifolds \cite{V}, by Sti\'enon and Xu in \cite{SX}
and by Antunes in \cite{A}. The relations of Nijenhuis structures on the double 
with Poisson-Nijenhuis structures have been 
discussed in these articles as well as in \cite{CGM}.
The deformation of generalized complex structures on manifolds was 
studied by Gualtieri in \cite{Gu}.

Our description of Nijenhuis structures and related concepts 
relies on the use of Roytenberg's
graded Poisson bracket on the
minimal symplectic realization of a Courant algebroid \cite{R}, and
on its interplay with the
big bracket \cite{R2002} \cite{yks2004} \cite{yksquasi} \cite{yksPM287}
when the Courant algebroid is 
the double of a Lie bialgebroid, and, in particular, in the case of
the generalized tangent bundle of a manifold.
We consider vector bundle endomorphisms that are skew-symmetric 
with respect to the fiberwise symmetric bilinear form of the Courant
algebroid, a natural assumption that expresses the infinitesimal
invariance of the symmetric bilinear form, and permits their inclusion
in computations in
the Poisson algebra of functions on the
minimal symplectic realization. 
We argue that, in the deformation theory of a Courant structure,
$\Theta$, by a
skew-symmetric tensor, $\N$, 
the decisive property is not the vanishing of the Nijenhuis torsion of
$\N$ but the property which we call `weak deforming' (Definition \ref{weakdef}).
When $\N^2$ is a scalar multiple of the identity, a condition that
appears repeatedly in \cite{CGM} and in \cite{nunes}, 
this condition is equivalent to
the `weak Nijenhuis' condition introduced in \cite{CGM},  $\N$ is a
weak Nijenhuis tensor if the Nijenhuis torsion of $\N$ is a cocycle
for the differential $\d_\Theta = \{ \Theta, \cdot ~\}$.
Our approach yields both new proofs of known results of \cite{CGM} and
\cite{nunes}, which we obtain with few or no computations, 
and several results which we believe
have not appeared elsewhere, especially Theorems \ref{CNS}, \ref{A3bialg} and
\ref{theoremweakdef}.

In Section \ref{section1} we recall results of \cite{CGM0} and
\cite{CGM} on Nijenhuis structures on Leibniz algebras and Leibniz algebroids.
In Section \ref{section2} we sketch the derived bracket approach to
Courant algebroids \cite{R} \cite{KW},
and we give a 
definition of irreducibility, adapted from \cite{G}.
In Section \ref{section3}, we study the properties of the Nijenhuis torsion
of a skew-symmetric endomorphism, $\N$, of a Courant algebroid, and we show
that the torsion has the usual properties of tensoriality and 
skew-symmetry in the special case where $\N^2$ 
is a scalar multiple of the identity (Theorem \ref{A3}).
In fact, on an irreducible Courant algebroid, any (skew-symmetric)
Nijenhuis tensor is proportional to a complex, paracomplex or subtangent
structure (Theorem \ref{courantirred}). 
We are thus led to a definition of `weak deforming tensors' which are
those tensors that 
generate infinitesimal deformations of Courant structures
(Theorem \ref{CNS}).

Section \ref{section4} 
deals with those Courant algebroids that are the double of a Lie
bialgebroid. In Theorem \ref{A3bialg}, 
we prove that, in the special case of an endomorphism, $N$,
of a Lie algebroid, $A$, whose square is a scalar multiple of the identity, the
torsion of the corresponding skew-symmetric endomorphism of its double
$A \oplus A^*$ is, in a suitable sense, the sum of the torsion of $N$ and
the torsion of its transpose.
Theorems \ref{biNij} and \ref{biNijconverse} are reformulations or
generalizations of results of \cite{CGM} and
\cite{nunes}.
Theorem \ref{bialg} deals with the deformation of Lie bialgebroids. 
In particular, in the case of a trivial Lie bialgebroid, a Nijenhuis
tensor on $A$ defines a weak deforming tensor for $A \oplus A^*$ 
(Theorem \ref{theoremweakdef}). 
Finally, in Section \ref{compatible}, we outline the role of
Poisson-Nijenhuis (or PN-) structures and
of  presymplectic-Nijenhuis (or $\Omega$N-) structures on a Lie
algebroid -- for which see, e.g.,
\cite{yksR} and references cited there -- 
in the deformation theory of the double of the Lie algebroid.
In Propositions \ref{PN} and \ref{OmegaN}, we prove that both
PN-structures and $\Omega$N-structures on a Lie
algebroid, $A$, define infinitesimal  
deformations of the double $A \oplus A^*$ of $A$.

The role of the Nijenhuis tensors\footnote{The early history of what
  is now called the Nijenhuis torsion of an endomorphism of the
  tangent bundle of a manifold is interesting and
  involved, and can be traced through
Paulette Libermann's article of 1955 in the {\it
  Bull. Soc. Math. France}, 
``Sur les structures presque
complexes et autres structures infinit\'esimales r\'eguli\`eres''.
Charles Ehresmann  defined the almost complex structures
in 1947 and developped their study in his address to the ICM at
Harvard, ``Sur les vari\'et\'es presque
complexes'' (1950). He defined the ``torsion''
  so that its vanishing was a necessary
  condition for these structures to be integrable, 
and he proved its vanishing to be 
also sufficient in
  the real analytic case. (In the smooth case, 
the sufficiency was proved much later,
in 1957, by Louis Nirenberg and A. Newlander.) 
This ``torsion'' was first 
defined as the torsion of a non canonically defined affine connexion
associated with the almost complex structure. It follows from the
reuslts of Libermann's doctoral thesis (1953) 
that this torsion is independent of the choice of the connexion and
that its expression in local coordinates could be given in terms of
the components
of the tensorial field itself. 
There was related work, some published and some unpublished,
by Beno Eckmann, Alfred 
Fr\"olicher, Kentaro Nomizu, as well as by Eugenio 
Calabi, Georges de Rham, 
Andr\'e Lichnerowicz and Jan A. Schouten. In her article, Libermann attributed to
Eckmann the proof that the vanishing of this
``torsion'' is equivalent to the condition 
$[X,Y] + J[JX,Y] +J [X,JY] -[JX,JY]=0$ for all tangent vector fields $X$
and $Y$. The more general
expression for the torsion of an endomorphim had been introduced by
Albert Nijenhuis in 1951 in his study of the integrability of the
distribution spanned by the eigenvectors of an endomorphism. This
torsion was later recognized as a special case of the bracket
introduced by Fr\"olicher and Nijenhuis in their articles of 1956 and 1958.}  
in the theory of Dirac pairs \cite{GD} \cite{D1987} \cite{D} 
that generalize the bihamil
tonian structures and have applications to
integrable systems \cite{BDSK}, and the theory of Dirac-Nijenhuis
structures \cite{CGM} \cite{nunes} \cite{H} 
will be the subject of further research.

\section{Nijenhuis structures on Leibniz algebras}\label{section1}

\subsection{Leibniz algebras}
A \emph{Leibniz algebra} (or \emph{Loday algebra}) 
is a vector space $L$ over a field $k$ of characteristic $0$,
equipped with a $k$-bilinear composition law, called the \emph{Leibniz bracket}, 
satisfying the Jacobi identity, 
\begin{equation}
u \circ (v \circ w) = (u \circ v) \circ w + v \circ (u \circ w),
\end{equation}
for all $u,v,w$ in $L$.
A Leibniz algebra with a skew-symmetric composition law is a Lie algebra.
The \emph{Leibniz cohomology}, which was defined by Loday \cite{L}, is
a generalization of the Chevalley-Eilenberg cohomology of
Lie algebras.

\subsection{Nijenhuis torsion}
Let $(L, \circ )$ be a Leibniz algebra. For an endomorphism $N$ of $L$, 
define 
\begin{equation}\label{axiom4}
u \circ_N v= N u  \circ v + u \circ Nv - N (u \circ v),
\end{equation}
and set 
\begin{equation}\label{5}
(T_\circ N) (u,v) = Nu \circ Nv - N(u \circ_N v).
\end{equation} 
Then 
$T_\circ N : L \times L \to
L$ is called the \emph{Nijenhuis torsion} or simply 
the \emph{torsion } of $N$, and $N$ is said to be a \emph{Nijenhuis
tensor} or a \emph{Nijenhuis structure} on $(L, \circ)$ if $T_\circ N =0$. 

\subsection{Deformations of Leibniz brackets}
A necessary and sufficient condition for $\circ_N$ to be a Leibniz bracket is  
that $T_\circ N$ be a Leibniz cocycle.
Then $\circ_N$ is a trivial infinitesimal deformation of~$\circ$. 
In particular, Nijenhuis tensors define trivial 
infinitesimal deformations of Leibniz brackets.
More precisely \cite{CGM}, 

\begin{proposition}
When $N$ is a Nijenhuis tensor on $(L, \circ)$,  

\noindent(i)
$\circ_N$ is a Leibniz bracket, 

\noindent(ii)
$N$ is a morphism of Leibniz algebras from $(L, \circ_N )$
 to $(L, \circ )$, and

\noindent(iii)
$\circ_N$ is compatible with $\circ$ in the sense that their sum
is a Leibniz bracket.
\end{proposition}

\subsection{Nijenhuis structures on Leibniz algebroids}
Leibniz algebroids are generalizations of Lie algebroids in which the
Lie bracket on the space of sections is only assumed to be a Leibniz
bracket. They are defined in \cite{Poncin}, 
where they are called {\it Loday algebroids}.
The definitions of Nijenhuis torsion and Nijenhuis structures
on Leibniz algebroids are similar to those in the case of Lie algebroids, since
they involve only the bracket and do not make use of the anchors. 

For a vector bundle endomorphism, $\N : E \to E$, of a Leibniz
algebroid over a manifold $M$,
we denote by the same
letter, $\N$, the map it induces on the sections of $E$. Then we define 
the bracket $\circ_\N$ and the torsion of $\N$ by formulas \eqref{axiom4}
and \eqref{5}. 
A vector bundle endomorphism is called a \emph{Nijenhuis tensor} or a
\emph{Nijenhuis structure} if its torsion vanishes. 
 
\section{Courant algebroids}\label{section2}

\subsection{The anchor and bracket as derived brackets}
We follow the approach of Roytenberg \cite{R}. 
Let $(E,\langle ~, ~ \rangle)$ be a vector bundle
equipped with a fiberwise symmetric bilinear form.
Here
we shall assume that $ \langle
~, ~ \rangle$ is non-degenerate. (Non-degeneracy is not
assumed in the definition of Courant algebroids in \cite{B} nor in
that of
Courant-Dorfman algebras in \cite{RLMP}.)
The minimal symplectic realization of $(E,\langle ~, ~ \rangle)$ is the bundle
$\widetilde E = j^!(T^*[2]E[1])$, where $j : E[1] \to E[1] \oplus E^*[1]$
is defined by $u \mapsto (u, \frac12 \langle u, \cdot \rangle)$, and
$j^!$ denotes the pull-back by $j$.
The injective map, $j$, is such that 
$ <ju, jv> = \langle u, v \rangle $, for all
$u, v \in E$, where
$<~,~>$ is the canonical fiberwise symmetric bilinear form on $E \oplus E^*$.

Let $\A$ be the graded algebra of functions on the minimal 
symplectic realization $\widetilde E$ of $E$, equipped with its canonical 
Poisson bracket of
degree $-2$, which we denote by $\{~,~\}$. 
Then,  $\A^0= C^\infty(M) $, $\A^1 = \Gamma E$, 
and for all sections $u, v$ of $E$,
\begin{equation}
\{u,v\} = \langle u, v\rangle.
\end{equation} 

A {\it Courant algebroid} structure on a vector bundle, $E$, over a manifold $M$,
equipped with a fiberwise non-degenerate 
symmetric bilinear form, $\langle ~, ~ \rangle$,
is defined by an element $\Theta \in \A^3$ such that 
\begin{equation}
 \{\Theta, \Theta\}=0.
\end{equation} 
The anchor, $\rho$,
and 
the bracket, $[~,~]$, 
are defined by
\begin{equation}
\rho(u) f = \{\{u, \Theta\},f\},
\end{equation} 
and 
\begin{equation}
[u,v] = \{\{u, \Theta\},v\},
\end{equation} 
for all sections $u, v$ of $E$, and $f \in C^\infty(M)$.
Thus, the anchor and bracket are viewed as the
derived brackets by $\Theta$ of the
canonical Poisson bracket of $\A$ restricted to
$\A^1 \times \A^0= \Gamma E \times C^\infty(M)$ and to
$\A^1 \times \A^1= \Gamma E \times \Gamma E$, respectively.
Bracket $[~, ~]$ is a Leibniz bracket on $\Gamma E$, 
called the \emph{Dorfman bracket}. 
Courant algebroids are examples of
Leibniz algebroids. 
The \emph{Courant bracket} 
is the skew-symmetrization of the Dorfman bracket.

The operator $\d_\Theta = \{\Theta, \cdot\}$ is a cohomology operator on
$\A$. 

We shall make use of the
relations \cite{R},
\begin{equation} \label{axiom1}
[u,v] + [v,u] = \partial \langle u, v \rangle,
\end{equation}
and
\begin{equation} \label{axiom2}
\langle  [u, v], w \rangle 
+ \langle v, [u,w] \rangle = \langle u, \partial \langle v,w
\rangle\rangle ,
\end{equation}
where $\partial : C^\infty(M) \to \Gamma E$ is defined by 
\begin{equation}
\langle u,  \partial f \rangle = \rho(u) \cdot f.
\end{equation}

A vector bundle endomorphism $\phi$ of $E$ is called
{\it symmetric} if  
\begin{equation}  
\langle \phi u,v \rangle = \langle u, \phi  v \rangle, 
\end{equation} 
for all  
$u,v \in E$. This condition is written 
\begin{equation}  
\phi = \, ^t\! \phi,
\end{equation} 
where $ ^t\!\phi $ is defined by 
$\langle \phi u,v \rangle = \langle u, \, ^t\!\phi  v\rangle$, for
all sections $u,v$ of $E$.

Before defining irreducible Courant algebroids, we consider the
following properties, where $\phi$ is a vector bundle endomorphism of $E$,

\noindent($P_1$) for all sections $u$ and $v$ of $E$, 
$$
[u, \phi v] = \phi [u,v] \quad {\mathrm{and}} \quad  [\phi u, v] = \phi [u,v],
$$
\noindent($P_2$)
for all sections $u$ and $v$ of $E$, 
$$
[u, \phi v] = \phi [u,v] \quad {\mathrm{and}} \quad  \phi \, \partial
\langle u, v\rangle =\partial \langle \phi u,v \rangle,
$$
\noindent($P_1'$) 
for all sections $u$ and $v$ of $E$, 
$$
[u, \phi v] = \phi [u,v] \quad {\mathrm{and}} \quad  [\phi u, u] = \phi [u,u].
$$

From relation \eqref{axiom1} and the fact that the base field is not
of characteristic~$2$, it is easy to prove the following
lemma.
\begin{lemma}\label{lemma0} 
Let $\phi$ be a vector bundle endomorphism of $(E,\langle
  ~,~ \rangle)$.
Properties ($P_1$) and ($P_2$) are equivalent. If  $\phi$ is symmetric, 
properties ($P_1$) and ($P_1'$) are equivalent. 
\end{lemma}

We adopt the following definition:
\begin{definition}\label{defcourantirred}
A Courant algebroid $(E,\langle
  ~,~ \rangle)$ is irreducible if any symmetric vector bundle
endomorphism $\phi$ of $E$ satisfying property ($P_1$) above
is proportional to the identity endomorphism, ${\mathrm{Id}}_E$, of $E$.
\end{definition}

Our definition is inspired by, but different from Grabowski's
definition in \cite{G} in which the endomorphisms are not
required
to be symmetric and irreducibility is defined by means of
property ($P_1'$). However, it follows from the lemma that 
any irreducible Courant algebroid in
the sense of \cite{G} is irreducible in the
sense of Definition \ref{defcourantirred}. 

Examples of irreducible 
Courant algebroids will be given in Section \ref{subsection4.1}.

\subsection{Tensors on $E$ and functions on $\widetilde E$}
Let $(E,\langle ~, ~ \rangle)$ be a vector bundle
equipped with a fiberwise non-degenerate symmetric bilinear form.
Any tensor on $E$ can be identified with a contravariant or a 
covariant tensor using
the symmetric bilinear form, and any skew-symmetric contravariant 
or covariant tensor, $t$, 
can be identified with a function $\widetilde t$ on $\widetilde E$.
A skew-symmetric contravariant 
$k$-tensor $t$ on $E$ can be identified with a function on $E^*[1]$, 
and therefore also with a function 
$\hat t$ on $E[1]$ by means of the symmetric bilinear form. The
function 
$\widetilde t$ is the pull-back of $\hat t$ under the projection
${\widetilde E} \to E[1]$.
We shall now describe these identifications by means of local coordinates.

Let $(e_a)$ be a local basis of sections of $E$ such that 
$\langle e_a, e_b \rangle$ is constant. Set $g_{ab}=
 \langle e_a, e_b \rangle$. If
$(q^i, \tau^a, p_i, \theta_a)$ are local
canonical coordinates on $T^*[2]E[1]$, then 
$(q^i, \tau^a, p_i)$ are local
coordinates on $\widetilde E$. Since, under the map $j$, $\theta_a = \frac12
 g_{ab} \tau^b$, the non-vanishing Poisson brackets of these coordinates are 
\begin{equation}  
\{q^i, p_j\} = \delta^{i}_j \quad {\mathrm{and}} \quad \{\tau^a,
\tau^b\} = g^{ab}.
\end{equation} 

To a skew-symmetric contravariant 
$k$-tensor, $t = t^{a_1a_2 \ldots a_k} e_{a_1} e_{a_2} \ldots
  e_{a_k}$, there corresponds
\begin{equation} 
\widetilde t = {\frac{1}{k!}} t^{a_1a_2 \ldots a_k}
  g_{a_1b_1}g_{a_2b_2} \ldots g_{a_kb_k}
  \tau^{b_1}\tau^{b_2}\ldots\tau^{b_k} \in \A^k.
\end{equation}
When no confusion can arise, we shall sometimes write $t$ for
$\widetilde t$.
For instance, a section $u=u^ae_a$ of $E$ is identified with the function
${\widetilde u} = g_{ab} u^a\tau^b \in \A^1$.

Let $\N: E \to E$ be an endomorphism of the vector bundle $E$ which preserves 
$\langle ~, ~ \rangle$ infinitesimally,
i.e., such that 
\begin{equation}  
\langle \N u,v \rangle + \langle u, \N  v \rangle = 0, 
\end{equation} 
for all  
$u,v \in E$. This condition is written 
\begin{equation}  
\N + \, ^t\! \N =0.
\end{equation} 
Such a map is called 
\emph{infinitesimally orthogonal} or \emph{skew-symmetric}.

\begin{remark}\rm{
Other authors \cite{CGM} \cite{G} \cite{A} call 
these endomorphisms \emph{orthogonal}. In fact, when $\N^2=\lambda \,
{\mathrm{Id}}_E$, condition $\N  \, ^t\! \N ={\mathrm{Id}}_E$ is equivalent to
$\N  -  \lambda \, ^t\! \N = 0$. In particular, for a generalized almost
complex structure, $\N^2= - {\mathrm{Id}}_E$, and the 
conditions $\N \, ^t\! \N ={\mathrm{Id}}_E$ and $\N + \, ^t\! \N =0$
are equivalent.}
\end{remark}

In local coordinates, if $\N (e_a)  = \N ^b_a
e_b $, the condition $\N + \, ^t\! \N =0$ is $\N ^b_a g_{bc} + \N ^b_c
g_{ba} = 0$. 
When it is satisfied, 
the associated contravariant tensor, with components $\N ^a_c g^{cb}$, is
skew-symmetric, and
$\widetilde \N  = \frac12 \N ^b_a g_{bc} \tau^a \tau^c \in \A^2$.
A short computation shows that 
\begin{equation}\label{Nu} 
{\N (u)} = \{u, \widetilde \N \},
\end{equation} 
for all sections $u$ of $E$. In fact, when $\N$ is a skew-symmetric
endomorphism, ${\widetilde \N}$ 
is the unique element in $\A^2$
satisfying \eqref{Nu}.
See \cite{Ath} for more general results on the correspondence between
tensors on $E$ and functions on $\widetilde E$.

\section{Nijenhuis and deforming tensors on Courant algebroids}\label{section3}

Let $(E, \langle ~, ~ \rangle, \Theta)$ be a Courant
algebroid over a manifold $M$, where $ \langle
~, ~ \rangle$ is the fiberwise non-degenerate symmetric bilinear form,  
and $\Theta \in \A^3$ determines the anchor, $\rho$, 
and the Leibniz bracket on sections, $[~,~]$.

\subsection{Nijenhuis torsion}

In what follows, we shall assume that $\N: E \to E$ is a skew-symmetric
vector bundle endomorphism.  This is a natural assumption
since skew-symmetry means that $\N$
leaves  $\langle  ~,~ \rangle$ infinitesimally invariant. 
As above, we define 
\begin{equation}\label{defbracket}
[u,v]_\N = [\N u,v]+[u,\N v]- \N [u,v].
\end{equation}

\begin{lemma} In terms of the Poisson bracket of $\A$,
 \begin{equation}\label{lemma1}
[u,v]_\N = \{\{{u}, \{\widetilde{\N}, \Theta\}\}, {v}\},
\end{equation}
for all $u, v \in \Gamma E \simeq  \A^1$.
\end{lemma}

\begin{proof}
The proof is an application of the Jacobi identity for the Poisson
bracket, formally identical to the proof of the analogous formula for
Lie algebroids. See, e.g., lemma 2 of \cite{yksR}.
\end{proof}

We now define the \emph{Nijenhuis torsion}, or simply the \emph{torsion},
$T_\Theta\N$, of $\N$, as in \eqref{5}, by
\begin{equation}\label{deftorsion}
(T_\Theta\N) (u,v) =  [\N u, \N v] - \N [u,v]_\N,
\end{equation}
for all sections $u, v$ of $E$. A skew-symmetric endomorphism 
whose torsion vanishes
is called a \emph{Nijenhuis tensor}. 

\begin{remark}\label{remarkpaired}
\rm{For an endomorphism that is not skew-symmetric, the
torsion can still be defined by \eqref{defbracket} and 
\eqref{deftorsion}, and we observe
that, if $\N'= \N + \kappa \, {\rm{Id}}_E$, where $\kappa$ is a scalar, then
$T_\Theta(\N') = T_\Theta\N$. 
Thus, those properties of the torsion that are
proved under the assumption that $\N$ is skew-symmetric but whose
proof does not utilize the Poisson bracket are also valid for
endomorphisms $\N' = \N + \kappa \, {\rm{Id}}_E$, which 
are characterized by the condition $\N' + \, ^t\!\N' = 2\kappa \,
{\rm{Id}}_E $. Such endomorphisms are called \emph{paired} in \cite{CGM}.}
\end{remark}

\begin{remark}\rm{
One can also define the torsion $T_C\N$ of an
endomorphism $\N$ with respect to the Courant
bracket, $[~,~]_C$, replacing the Dorfman bracket by its
skew-symmetrization in the preceding formulas. 
The relation between the two torsions is
\begin{equation}\label{onehalf}
(T_C\N)(u,v) = \frac12 ((T_\Theta\N)(u,v) - (T_\Theta\N)(v,u)),
\end{equation}
while, for a skew-symmetric tensor $\N$,
\begin{equation}\label{16}
\begin{array}{lll}
&   & (T_C\N)(u,v)  -  (T_\Theta\N) (u,v)\\
& = &  \frac12 \left(- \partial  \langle \N u, \N v \rangle  
+ \N \partial \langle \N u, v \rangle + \N \partial \langle u, \N v
\rangle
- \N^2 \partial  \langle u, v \rangle \right)\\
& = &
 \frac12 \left( \partial  \langle  u, \N^2 v \rangle  
- \N^2 \partial  \langle u, v \rangle \right).
 \end{array}
\end{equation}
If $\N^2$ is a
scalar multiple of the identity of $E$, both torsions, $T_\Theta\N$ and $T_C\N$, coincide.}
\end{remark}

\subsection{Properties of the torsion}
For ease of exposition, we introduce the following definition from
\cite{V} (see also \cite{A}),
where `cps' stands for `complex, paracomplex or subtangent'.
\begin{definition}
A skew-symmetric endomorphism $\N$ of a Courant algebroid~$E$ such that 
$\N^2 = \lambda \, \mathrm{Id}_E$, where 
$ \lambda = - 1, \, 1$ or $0$, is called a \emph{generalized almost
  cps structure} on $E$. A generalized almost cps structure 
is called a \emph{generalized cps structure} if its
torsion vanishes.
\end{definition}

When $(E, \langle ~, ~ \rangle, \Theta)$ 
is a Courant algebroid, the torsion $T_\Theta\N$ of 
a skew-symmetric endomorphism $\N$ of $E$ 
is a map from $\Gamma E \times \Gamma E$ to
$\Gamma E$. 
Unlike the case of Lie algebroids, $T_\Theta\N$ is not in general 
$C^\infty(M)$-linear in both arguments, and in general not
skew-symmetric. 

\medskip

\noindent{\bf Linearity.}
It is clear that 
\begin{equation}
[u,fv]  =  f [u,v] + (\rho(u) \cdot f) v
\end{equation}
and 
\begin{equation}
[ f u , v ] = - [v,f u] + \partial \langle f u, v \rangle =
f [u,v] - (\rho(v) \cdot f) u + \langle u, v \rangle \partial f.
\end{equation}
Whence,
\begin{equation}
({T}_\Theta{\N})(u,fv) = f ({T}_\Theta{\N})(u,v), 
\end{equation}
and
\begin{equation}\label{17}
({T}_\Theta{\N})(fu,v) = f ({T}_\Theta{\N})(u,v) +
\langle u, v \rangle
\N^2(\partial f) -  \langle  u, \N^2 v \rangle \partial f.
\end{equation}
In fact, since $\N$ is skew-symmetric,
\begin{equation*}
\begin{array}{lll}
& & ({T}_\Theta{\N})(fu,v)  - f ({T}_\Theta{\N})(u,v)\\
& = & (\langle \N u,  v \rangle + \langle  u, \N v \rangle) \partial f
+  \langle \N u, \N v \rangle \partial f +  \langle u, v \rangle
\N^2(\partial f) \\
& = &  \langle u, v \rangle
\N^2(\partial f) -  \langle  u, \N^2 v \rangle \partial f.
\end{array}
\end{equation*}

\noindent{\bf Skew-symmetry.}
Again using the fact that $\N$ is skew-symmetric, we obtain
\begin{equation}\label{18}
({T}_\Theta{\N})(u,v)  +  ({T}_\Theta{\N})(v,u) \\
 =
\N^2 \partial \langle  u, v \rangle
-  \partial \langle  u, \N^2 v \rangle  
\end{equation}
since 
$$
({T}_\Theta{\N})(u,v)  +  ({T}_\Theta{\N})(v,u) =
\partial \langle \N u,  \N v \rangle - \N 
(\partial \langle  \N u,  v \rangle
+ \partial \langle  u,  \N v \rangle)
+ \N^2 \partial \langle  u, v \rangle.
$$

\begin{remark}
\rm{Equation \eqref{18} can alternatively be derived from 
\eqref{onehalf} and \eqref{16}.}
\end{remark}

\noindent{\bf Associated $3$-tensor.}
In order to determine whether ${T}_\Theta{\N}$ determines a
skew-symmetric covariant $3$-tensor, we use 
the skew-symmetry of $\N$  and relation \eqref{axiom2} to obtain 
\begin{equation}\label{20} 
\langle ({T}_\Theta{\N})(u,v), w \rangle
+ \langle ({T}_\Theta{\N})(u,w), v \rangle
= \langle  \N^2 [u, w] - [u, \N^2w], v \rangle .
\end{equation}

Equations \eqref{17}, \eqref{18} and \eqref{20} show that, when 
$\N^2 = \lambda \, \mathrm{Id}_E$, the torsion ${T}_\Theta{\N}$ of
$\N$ is
$C^\infty(M)$-linear in both arguments and skew-symmetric, and 
defines a skew-symmetric covariant $3$-tensor, 
$\widetilde{{T}_\Theta{\N}}$, on $E$ by
\begin{equation}
\widetilde{{T}_\Theta{\N}}(u,v,w) 
= \langle ({T}_\Theta{\N})(u,v), w \rangle. 
\end{equation}
More precisely,

\begin{theorem}\label{A3}
Assume that $\N$ is proportional to a generalized almost cps
structure on a Courant
algebroid, $(E, \langle ~, ~ \rangle, \Theta)$.

\noindent (i) The torsion, ${T}_\Theta{\N}$, of $\N$ is
$C^\infty(M)$-linear in both arguments and skew-symmetric, and it 
defines an element $\widetilde{{T}_\Theta{\N}} 
\in \A^3$.

\noindent (ii) For all sections $u,v$ of $E$,
\begin{equation}\label{torsion0}
{{({T}_\Theta{\N})(u,v)}} =
  \{\{{u},\widetilde{{T}_\Theta{\N}}\},
{v}\}.
\end{equation}

\noindent (iii) Set $\N^2 = \lambda \,
\mathrm{Id}_E$, for a real number $\lambda$. Then
\begin{equation}\label{torsioncps} 
\widetilde{{T}_\Theta{\N}} = -
\frac12 (\{\{{\widetilde \N}, \Theta\},{\widetilde \N}\} + \lambda
\Theta).
\end{equation}
\end{theorem}
\begin{proof}
Formulas \eqref{torsion0} and \eqref{torsioncps} follow from
\eqref{Nu} and \eqref{lemma1}, and
the use of the Jacobi identity for $\{~,~\}$.
\end{proof}

In addition, in view of Definition \ref{defcourantirred} and Lemma
\ref{lemma0}, 
from relations \eqref{20} and \eqref{18}, we obtain immediately,
\begin{theorem}\label{courantirred}
 If $\N$ is a Nijenhuis tensor on $E$, then 
$  \N^2 [u, v] = [u, \N^2v]$ and $\N^2 \partial \langle  u, v \rangle
= \partial \langle  u, \N^2 v \rangle$, for all sections $u,v$ of $E$.
If $E$ is irreducible, any 
Nijenhuis tensor on $E$ is proportional to a generalized cps structure.
\end{theorem}

Formula \eqref{torsioncps} was first stated in 
corollary 3 of \cite{G}.
A result equivalent to Theorem \ref{courantirred} was proved in \cite{CGM}
(theorem 5). 

\subsection{Deformations of Courant algebroids}
As above, we shall consider skew-symmetric endomorphisms
$\N$ of $(E, \langle  ~,~ \rangle)$ exclusively. 
In fact, tensors with vanishing
Nijenhuis torsion do not in general define trivial 
infinitesimal deformations of the Dorfman bracket
of a Courant algebroid, unless they have additional properties such
as being proportional to a generalized almost cps structure. We are 
thus led to introduce the following definitions.
\begin{definition}\label{weakdef} 
A skew-symmetric endomorphism
$\N$ of a Courant algebroid $(E, \langle  ~,~ \rangle, \Theta)$
is called a

(i) \emph{weak deforming tensor} for $\Theta$ if $\{\{{\widetilde \N},
\Theta\},{\widetilde \N}\}$ is a $\d_\Theta$-cocycle,

(ii) \emph{deforming tensor} for $\Theta$ 
if $\{\{{\widetilde \N},
\Theta\},{\widetilde \N}\}$ is a scalar multiple of $\Theta$,
\end{definition} 

This terminology is justified by the fact that,
because $\d_\Theta \Theta = \{\Theta, \Theta \} = 0$, 
any deforming tensor is a weak deforming tensor.
Theorem \ref{CNS} below is further justification for the terms that we have
introduced.

\begin{theorem}\label{CNS}
Let $\N$ be a skew-symmetric endomorphism of a Courant algebroid $(E,
\langle  ~,~ \rangle, \Theta)$. Then 
$\{{\widetilde \N},\Theta\}$ is a Courant algebroid structure on $(E,
\langle  ~,~ \rangle)$ if and only if ${\N}$ is 
a weak deforming tensor for $\Theta$.
\end{theorem}
\begin{proof} The theorem follows from the fact that, by the Jacobi
identity,
$$ \{\{{\widetilde \N},
\Theta\},
\{{\widetilde \N},
\Theta\} \} = \{\Theta, \{\{{\widetilde \N},
\Theta\},{\widetilde \N}\}\},
$$
so $\{{\widetilde \N}, \Theta\}$ is a Courant algebroid structure on
$(E, \langle  ~,~ \rangle )$
if and only if $\{\{{\widetilde \N},
\Theta\},{\widetilde \N}\}$ is a $\d_\Theta$-cocycle.
\end{proof}

\begin{remark}\rm{
The condition $\{\{{\widetilde \N},
\Theta\},{\widetilde \N}\} = 0$ is sufficient for $\{{\widetilde \N},
\Theta\}$ to be 
a Courant algebroid structure. It expresses the vanishing of the Maurer-Cartan
element $[{\widetilde \N},
{\widetilde \N}]^\Theta$ in the differential graded Leibniz-Poisson algebra
$(\A, [~,~]^\Theta, 0)$, where $[~,~]^\Theta$ is the derived bracket
of the Poisson bracket $\{~,~\}$ by the odd interior 
derivation of square $0$, $\{\Theta,\cdot\}$.}
\end{remark}

When $\N$ is a weak deforming tensor, the
Courant algebroid structure $\{{\widetilde \N}, \Theta\}$ on 
$(E, \langle ~, ~ \rangle)$ is compatible
with $\Theta$. In fact $\{\Theta + \{{\widetilde \N}, \Theta\}, \Theta
+ \{{\widetilde
  \N}, \Theta\}\}$ vanishes.

\medskip

Now, the condition $\{\Theta, T_\Theta\N\}=0$ makes sense only if
$T_\Theta\N$ is an element of $\A^3$. If $E$ is irreducible this is
the case if and only if $\N$ is proportional to a generalized almost
cps structure, whence the following definition.

\begin{definition}\label{weaknij} 
If $\N$ is proportional to a generalized almost cps structure and if 
$T_\Theta \N$ is a $\d_\Theta$-cocycle, $\N$ 
is called a \emph{weak Nijenhuis tensor}. 
\end{definition} 

In the case of tensors proportional to a generalized almost cps structure, 
we observe the following implications and equivalence: 

\noindent $\bullet$ 
A Nijenhuis tensor is a weak Nijenhuis tensor.

\noindent $\bullet$ A Nijenhuis tensor is a deforming
tensor, and therefore also a weak deforming tensor.

\noindent $\bullet$ A tensor is weak Nijenhuis if and only if it is
weak deforming. 

$$
\begin{matrix}
\mathrm{Nijenhuis} & \Rightarrow & \mathrm{weak \, \, Nijenhuis}\\

\Downarrow & & \Updownarrow\\

\mathrm{deforming} & \Rightarrow & \mathrm{weak \, \, deforming}
\end{matrix}
$$

We can now state a corollary of Theorem \ref{CNS} 
concerning the special case of those endomorphisms 
whose square is a scalar multiple
of the identity.
\begin{corollary}\label{corollaryweak}
Let $\N$ be a skew-symmetric endomorphism of a Courant algebroid $(E,
\langle  ~,~ \rangle, \Theta)$, proportional to a generalized almost cps
structure. 
Then 
$\{{\widetilde \N},\Theta\}$ is a Courant algebroid structure on $(E,
\langle  ~,~ \rangle)$ if and only if ${\N}$ is 
a weak Nijenhuis tensor.
\end{corollary}

While the compatibility of $\Theta$ and $\{\widetilde \N, \Theta\}$ is 
satisfied as soon as $\N$ is weak deforming, 
it is the vanishing of the torsion which implies a morphism
property of $\N$. If we recall that a generalized almost cps structure
is a generalized cps
structure if and only if its torsion
vanishes, we can state,
\begin{proposition}\label{propmorphism}
Let $\N$ be a skew-symmetric endomorphism of $E$ proportional to a generalized
almost cps structure. Then 
$\N$ is a morphism of Courant algebroids from 
$(E, \langle ~, ~ \rangle, \{{\widetilde \N}, \Theta\})$ to 
$(E, \langle ~, ~ \rangle,\Theta)$ if and only if $\N$ is
proportional to a generalized cps structure.
\end{proposition}

Corollary \ref{corollaryweak} and 
Proposition \ref{propmorphism} imply and are implied by results to
be found in \cite{CGM} and \cite{G}.

\section{The case of the double of a Lie bialgebroid}\label{section4}

\subsection{The double of a Lie bialgebroid}\label{subsection4.1}
Let $A$ be a vector
bundle and let $E = A \oplus A^*$ be equipped with the canonical
symmetric bilinear form $\langle ~ , ~ \rangle$. 
The minimal symplectic realization of $E$ is $\widetilde E
= T^*[2]A[1]$, and the canonical Poisson bracket of $\A$ coincides with
the big bracket (for which see \cite{R2002} \cite{yksquasi}
\cite{yksPM287}), which we also denote by $\{~,~\}$.

Let $((A, \mu),(A^*,\gamma))$ be a Lie bialgebroid over a manifold
$M$. Then $\mu$ and $\gamma$ are elements of $\A^3$ satisfying 
$\{\mu,\mu\}= \{\mu,\gamma\} = \{\gamma, \gamma\} =0$.
The canonical
symmetric bilinear form and $\Theta = \mu +
\gamma$ turn $E = A \oplus A^*$ into a Courant algebroid called the
\emph{double} of $((A, \mu),(A^*,\gamma))$. See
\cite{LWX}. 
In Section \ref{trivial}, we shall consider the case 
$\gamma = 0$, in which case the Lie bialgebroid is called \emph{trivial}.
In particular if $A=TM$ equipped with the identity endomorphism as
anchor and the Lie bracket of vector fields
and if $\gamma =0$, 
then $TM \oplus T^*M$ is the
\emph{standard Courant algebroid} or \emph{generalized tangent bundle}
  of $M$.

\begin{definition}
A Lie algebroid $A$ is called \emph{irreducible} if any vector bundle 
endomorphism $\psi$ of $A$ satisfying $\psi[X,Y] = [X, \psi Y]$, for
all sections $X, Y$ of $A$, is proportional to the identity,
${\mathrm{Id}}_A$, of $A$. 
\end{definition}
It is proved in \cite{CGM} that the tangent
bundle of any connected manifold is an irreducible Lie algebroid.
We can now give examples of irreducible Courant algebroids.

When $(A,\mu)$ is an irreducible Lie algebroid over a connected 
manifold, the double of the trivial Lie bialgebroid $((A, \mu),(A^*,0))$ is
an irreducible Courant algebroid in the sense of Definition
\ref{defcourantirred}. In particular, the generalized tangent bundle of a
connected manifold is an irreducible Courant algebroid.
To prove this statement, we write a symmetric endomorphism of  $A \oplus A^*$
as $\phi = \begin{pmatrix} \psi & \alpha\\ \beta & \,^t \!
\psi \end{pmatrix}$, where $\psi$ is an endomorphism of $A$, and
$\alpha: A^* \to A$ and $\beta: A \to A^*$ are symmetric. We then 
express the conditions $\phi[u,v] = [\phi u,v]$ and 
$\phi [u,v] = [u, \phi v]$ for, successively, $u=X$ and $v=Y$, then
$u=X$ and $v=\eta$, then $u=\xi$ and $v=Y$, then $u=\xi$ and $v=\eta$,
where $X, Y \in \Gamma A$ and $\xi, \eta \in \Gamma(A^*)$.
By the irreducibility of $(A,\mu)$, we find that $\psi$ is a constant
multiple of the identity of $A$, then that $\alpha$ and $\beta$ must vanish.

\subsection{Generalized almost cps structures on $A \oplus A^*$}
Any vector bundle endomorphism of $E = A \oplus A^*$ is of the form 
$\N = \begin{pmatrix}N & \pi\\ \omega & N' \end{pmatrix}$, where 
$N : A \to A$, $N' : A^* \to A^*$, $\pi : A^* \to A$ and 
$\omega : A \to A^*$. The endomorphism $\N$ is 
skew-symmetric if and only if $N' = - \, ^t \!N$, $\pi$ is a bivector
on $A$, and $\omega$ is a $2$-form on $A$.

The conditions for $\N^2 = \lambda \, {\mathrm {Id}_E}$ are (i) $N\pi$ is a
bivector, (ii) $\omega N$ is a $2$-form and (iii) $N^2 + \pi \omega =
\lambda \, {\mathrm{Id}}_A$. A sufficient condition
for (iii) is that $N^2$ be a scalar multiple of the
identity and that $\pi$ and $\omega$ be inverses of one another, or
that $\pi= 0 $, or that $\omega = 0$.
See \cite{SX} \cite{A} for the general case and its interpretation in terms of
Poisson quasi-Nijenhuis structures.
\subsection{Tensors on $A$ and functions on $T^*[2]A[1]$}
Let $A$ be a vector bundle.
We show that skew-symmetric tensors on $A$ can be identified 
with elements of $\A$, the graded algebra of functions on $T^*[2]A[1]$.

A tensor $t \in A^* \otimes A$
can be considered as an element in $(A \oplus A^*) \otimes (A \oplus
A^*)$ by setting
\begin{equation}
{t}(  X + \xi ; Y + \eta) = \langle t(X), \eta \rangle,
\end{equation}
and, because $A \oplus A^*$ is self-dual, $t$ can be skew-symmetrized
into the element $\widetilde{t}$ in $\wedge^2(A \oplus A^*)$ such that
\begin{equation}
 \widetilde{t}(X+\xi,Y+\eta) =
\langle t(X), \eta \rangle
- \langle t(Y), \xi \rangle,
\end{equation} 
for all  $X, Y \in A$ and  $\xi, \eta \in A^*$. 
The map induced on sections of $A \oplus A^*$ by  $\widetilde{t}$
is the element in $\A$ that corresponds to $t$.

In other words, if $N$ is a vector bundle endomorphism of $A$, 
considered as an element in $A^* \otimes A$, then
the
skew-symmetric endomorphism $\N$ of $A \oplus A^*$ defined by $N$ is 
such that
\begin{equation}\label{scriptN}
\N (X+\xi) 
 = NX - \, ^t\! N \xi,
\end{equation} 
and as in \eqref{Nu},
\begin{equation}\label{tildescriptN}
\N (X+\xi) = \{X + \xi,  \widetilde{\N}\}. 
\end{equation} 

We can also skew-symmetrize higher-order tensors. 
A tensor $t \! \in \! \wedge^2 A^* \otimes A$
can be considered as an element in $\wedge^2(A \oplus A^*) \otimes (A \oplus
A^*)$ by setting
\begin{equation}
{t}(X + \xi, Y +\eta ; Z +\zeta) = \langle t(X,Y), \zeta \rangle,
\end{equation}
in which case $ \widetilde{t}$ is
the element in $\wedge^3(A \oplus A^*)$ such that
\begin{equation}
 \widetilde{t}(X+\xi,Y+\eta,Z+\zeta) =
\langle t(X,Y), \zeta \rangle
+ \langle t(Y,Z), \xi \rangle + \langle t(Z, X), \eta \rangle,
\end{equation} 
for all  $X, Y , Z \in A$ and  $\xi, \eta, \zeta \in A^*$. 
The map induced on sections of $A \oplus A^*$ by  $\widetilde{t}$
is the element 
in $\A$ that corresponds to $t$. Then
\begin{equation}
 \widetilde{t}(X+\xi,Y+\eta,Z+\zeta) = 
\{\{\{X+\xi, {\widetilde{t}} \, \}, Y+\eta\}, Z +\zeta\} .
\end{equation}

Similarly, if $t \in \wedge^2 A \otimes A^*$, then $\widetilde t 
\in \wedge^3(A \oplus A^*) $ is defined by
\begin{equation}
\widetilde t(X+\xi,Y+\eta,Z+\zeta) = \langle t(\xi,\eta),Z\rangle +
\langle t(\eta,
\zeta),X\rangle 
+ \langle t(\zeta,\xi ),Y\rangle.
\end{equation} 

Let $(e_\alpha)$ be a local basis of sections of $A$ and let 
$(\epsilon^\alpha)$ be the dual basis. Let $(x^i,\tau^\alpha, p_i, 
\theta_\alpha)$ be local canonical coordinates on $T^*[2]A[1]$. If
$N=N^\alpha_\beta \epsilon^\beta e_\alpha$, then 
$\widetilde \N = N^\alpha_\beta \tau^\beta \theta_\alpha$. 
If $t=\frac12 t^\alpha_{\beta\gamma} \epsilon^\beta \epsilon^\gamma 
e_\alpha$, then 
$\widetilde t = \frac12 t^\alpha_{\beta\gamma} \tau^\beta \tau^\gamma
\theta_\alpha$.
As elements in $\A^2$, $N$ and $\widetilde \N$ are identified, and as
elements in $\A^3$, $t$ and $\widetilde t$ are identified.

\subsection{The double of deformed brackets}
Let  $(A, \mu)$ and $(A^*,\gamma)$ be Lie algebroids, with brackets
denoted by  $[~,~]^{\mu}$ and $[~,~]^{\gamma}$. Let 
$N:A \to A$ be a vector bundle endomorphism. Set
$\{N,\mu\}=\mu_N$ and $\{N,\gamma\} = \gamma_{-\, ^t\!N}$. The
associated brackets are  $[~,~]^{\mu_N} = [~,~]^\mu_N$ on $A$, and 
$[~,~]^{\gamma_{-\, ^t\!N}} = [~,~]^\gamma_{-\, ^t\!N}$
on $A^*$, which are not necessarily Lie brackets. 
We consider the bracket on $A \oplus A^*$ defined by $\mu + \gamma$
(respectively, $\mu_N + \gamma_{-\, ^t\!N}$), 
which we call the \emph{double} of brackets $\mu$ and $\gamma$
(respectively, $[~,~]^{\mu}_N$ and $[~,~]^{\gamma}_{-\, ^t\!N}$).  
\begin{proposition}\label{propdouble}
Let  $(A, \mu)$ and $(A^*,\gamma)$ be Lie algebroids and let $[~,~]^{\mu+\gamma}$ be
the double of brackets $[~,~ ]^\mu $
and $[~,~]^\gamma$ in $A \oplus A^*$. 
Let $N$ be a vector bundle
endomorphism of $A$,
and let $\N$ be the associated skew-symmetric endomorphism 
of $A \oplus A^*$ defined by \eqref{scriptN}.
Then bracket $[~,~]^{\mu+\gamma}_\N$ is the double of 
brackets $[~,~]^\mu_N$ and $[~,~]^\gamma_{-\, ^t\!N}$.
\end{proposition}
\begin{proof}
It is clear that 
\begin{equation}\label{double}
\{ \widetilde\N, {\mu} + {\gamma}\} =
\{ \widetilde\N, {\mu}\}
+\{ \widetilde \N,{\gamma}\} = {\{N,\mu\}} +
{\{N, \gamma\}}.
\end{equation}
The result follows.
\end{proof}
\begin{remark}\rm{
The result of the proposition is valid more generally, 
independently of the assumptions $\{\mu,\mu\} =0$
and $\{\gamma, \gamma\}=0$ which express the fact that $(A, \mu)$ and
$(A^*,\gamma)$ are Lie algebroids.}
\end{remark}
\subsection{Torsion of $\N$ in the case of the double of a Lie
  bialgebroid}\label{sectiondouble}

When $(A, \mu)$ and $(A^*,\gamma)$ are Lie
algebroids, if the torsion 
${ T}_{\mu +\gamma}\N$ of $\N$ defines an element in $\A^3$, 
we can compare 
$\widetilde{{ T}_{\mu +\gamma}\N}$ with the sum of 
the elements in $\Gamma (\wedge^3(A \oplus A^*)) \subset \A ^3$ 
defined by the torsion $T_\mu N$ of $N$ and the
torsion $T_\gamma \,^t\!N$ of
$^t \! N$. 
 
\begin{definition}
A vector bundle endomorphism $N$ of a vector bundle $A$ 
is called an \emph{almost cps
structure} if $N^2 = \epsilon \, {\mathrm{Id_A}}$, with $\epsilon = -
1, \, 1$ or $0$. On a 
Lie algebroid $(A,\mu)$, an almost cps
structure is called a \emph{cps structure} if, in addition, the Nijenhuis
torsion $T_\mu N$ of $N$ vanishes. 
\end{definition}

\begin{theorem}\label{A3bialg}
Let $((A, \mu),(A^*,\gamma))$ be a Lie bialgebroid.
Let  $N:A \to A$ be a vector bundle endomorphism, and let 
$\N$ be the skew-symmetric endomorphism of $A \oplus A^*$ with matrix
$\begin{pmatrix}N & 0\\0 & - \, ^t \!N\end{pmatrix}$.

\noindent (i) The element $\{\{\widetilde \N, {\mu} +
{\gamma} \},
\widetilde{\N}\}$ is in $\A^3$ and is equal to $\{\{ {N}, \mu + \gamma\}, N \}$.

\noindent (ii) 
If $N$ is
proportional to an almost cps structure on $A$, then
\begin{equation}\label{eqdouble}
\widetilde{{T}_{\mu +\gamma}\N}
=\widetilde{T_\mu N} + \widetilde{T_\gamma \, ^tN}.
\end{equation} 
The explicit form of equation \eqref{eqdouble} is
\begin{equation}
\begin{array}{lll}
({T}_{\mu +\gamma}\N)(X+\xi, Y+\eta, Z+\zeta)\\ =
(T_\mu N)(X,Y, \zeta)+  (T_\mu N)(Y,Z, \xi)+  (T_\mu N)(Z, X, \eta)\\
+ (T_\gamma \, ^t \!N)(\xi, \eta, Z)+  (T_\gamma \, ^t \!N)(\eta, \zeta, X)+
(T_\gamma \, ^t \!N)(\zeta, \xi, Y), 
\end{array}
\end{equation}
for all sections $X+\xi, Y+\eta, Z+\zeta$ of $A\oplus A^*$.
\end{theorem}
\begin{proof} 
The proof of (i) is based on the remarks
that 
$\{\{ \widetilde{\N}, \mu \}, \widetilde{\N}\}$ 
(respectively, $\{\{ \widetilde{\N}, {\gamma} \}, \widetilde{\N}\}$)
is equal to the element 
$\{\{ {N}, \mu \}, {N}\}$ 
(respectively, $\{\{ {N}, \gamma\}, {N}\}$) of $\A^3$. Therefore, by 
Proposition \ref{propdouble},
$\{\{ \widetilde\N, {\mu} +{\gamma} \}, \widetilde{\N}\}$ is equal to
$\{\{ {N}, \mu + \gamma\}, N \}$.

Since, when $N^2 = \lambda \, {\mathrm{Id}}_A$, by formula \eqref{torsioncps},
$$T_\mu N = - \frac12(\{\{ {N}, \mu \}, {N}\} + \lambda \mu), \quad
T_\gamma \, ^t\!N = - \frac12(\{\{ {N}, \gamma \}, {N}\} + \lambda \gamma), 
$$
and 
$$
\widetilde{T_{\mu+\gamma} {\N}} = 
- \frac12(\{\{\widetilde{\N}, {\mu} + {\gamma}\},
 \widetilde{\N}\} 
+\lambda ({\mu} + {\gamma})),
$$ 
the result of (ii) follows.
\end{proof}

\begin{remark}\rm{
The result of the theorem is valid more generally, 
independently of the assumptions $\{\mu,\mu\} = \{\gamma, \gamma\}=
\{\mu,\gamma\} =0$ which express the fact that $((A, \mu),(A^*,\gamma))$
is a Lie bialgebroid.}
\end{remark}

\begin{remark}\rm{
Since the Dorfman bracket on $\Gamma(A \oplus A^*)$ reduces to
$[~,~]^\mu$ on $\Gamma A$ and to
$[~,~]^\gamma$ on $\Gamma A^*$, it is clear that, for any endomorphism $N$ of $A$,
\begin{equation}
(T_{\mu + \gamma} \N)_{|A} = T_\mu N \quad \mathrm{and} \quad (T_{\mu + \gamma}
\N)_{|A^*} = T_\gamma \, ^t\! N.
\end{equation}}
\end{remark}

As a consequence of Theorem \ref{A3bialg} (ii),
we obtain, 

\begin{theorem}\label{biNij}
Let $((A, \mu),(A^*,\gamma))$ be a Lie bialgebroid.
Let $N$ be proportional to an almost cps structure on $A$, and assume
that $N$ is a Nijenhuis tensor for $(A,\mu)$ and $(A^*,\gamma)$, 
i.e., $T_\mu N = 0$ and $T_\gamma \, ^t\!N = 0$. Then $N$ gives rise to a 
Nijenhuis tensor $\N = \begin{pmatrix}N & 0\\0 & - \, ^t \!N\end{pmatrix}$ 
for the Courant algebroid 
$(A \oplus A^*, \mu + \gamma)$.
\end{theorem}

For any scalar $\kappa$,
$T_\mu (N + \kappa \, {\rm{Id}}_A) = T_\mu N$, 
$T_\gamma (\,^t\!N + \kappa \, {\rm{Id}}_{A^*}) = T_\gamma (\,^t\!N)$
and $T_{\mu + \gamma} (\N + \kappa \, {\rm{Id}}_{A \oplus A^*}) = T_{\mu +
  \gamma} \N$, and therefore 
Theorem \ref{biNij} is also valid in the slightly more
general case of
$\N + \kappa \, {\rm{Id}}_E$, where $\N$ is skew-symmetric and $\kappa$ is a
scalar. But there is no analogous statement for more general
non skew-symmetric endomorphisms of $A \oplus A^*$. (In theorem 4.1 of
\cite{nunes}, because of a change of notation in the course of the
proof, the assumption $N^2 = \lambda_2$ should be replaced by
$(N - \lambda_1)^2 = \lambda_2$. With this modification, that theorem
is equivalent to the preceding generalized form of Theorem  \ref{biNij}.)

\medskip

We obtain the following converse of Theorem \ref{biNij} as a particular
case of Theorem \ref{courantirred}.
\begin{theorem}\label{biNijconverse}
Let $((A, \mu),(A^*,\gamma))$ be a Lie bialgebroid such that $A \oplus
A^*$ is an irreducible Courant algebroid. If 
$\N = \begin{pmatrix} N & 0 \\ 0 & - \,^t\!N 
\end{pmatrix}$ is a Nijenhuis tensor for $A \oplus A^*$,
then $N$ is
proportional to a cps
structure on $A$, and $\,^t\!N$ is proportional to a cps
structure on $A^*$.
\end{theorem}

We now outline an alternate, computational 
proof of Theorem \ref{A3bialg} 
that does not use the Poisson bracket of $\A$. This longer proof consists of
computing the vector part and the form part of $({T}_{\mu
  +\gamma}\N)(X, Y)$,
$({T}_{\mu +\gamma}\N)(\xi, \eta)$, $({T}_{\mu +\gamma}\N)(X, \eta)$
and $({T}_{\mu +\gamma}\N)(\xi, Y)$, and then the duality product of each with
$Z$ or $\zeta$. It utilizes the definitions 
$$ [X,Y] =[X,Y]^\mu, \,
[X,\eta] = - i_\eta \d_\gamma X + {\mathcal L}^\mu_X \eta,  \,
 [\xi, Y] = {\mathcal L}^\gamma_\xi Y - i_Y \d_\mu \xi, \,
[\xi, \eta] =[\xi, \eta]^\gamma, 
$$
and $\N X = NX$, $\N \xi = - \, ^t \! N \xi$. 
Clearly
\begin{equation}\label{01}
\langle  (T_{\mu+ \gamma} \N) (X,Y) \, , \, Z + \zeta \rangle 
= \langle {(T_{\mu} N)} (X, Y), \zeta\rangle  
\end{equation}
and
\begin{equation}\label{02}\langle  (T_{\mu+ \gamma} \N) (\xi,\eta) \, ,
  \, Z + \zeta \rangle 
= \langle {(T_{\gamma} \, ^t \!N)} (\xi, \eta), Z \rangle .
\end{equation}
One finds, after a computation,
\begin{equation}\label{1}
\langle  (T_{\mu+ \gamma} \N) (X,\eta) \, , \, Z \rangle  = 
\langle  (T_{\mu} N) (Z,X) +  [N^2 Z,X]^\mu -
N^2 [Z,X]^\mu \, , \, \eta \rangle ,
\end{equation}
\begin{equation}\label{2}
\langle  (T_{\mu+ \gamma} \N) (X,\eta) \, , \, \zeta \rangle 
= 
\langle (T_{\gamma} \, ^t\! N) (\eta,\zeta) \, , \, X \rangle + 
\d_\gamma (N^2X)(\eta,\zeta) - (\d_\gamma X)(\eta, \, ^t\!N^2\zeta).
\end{equation}
Similarly, one finds
\begin{equation}\label{3}
\langle  (T_{\mu+ \gamma} \N) (\xi,Y) \, , \, \zeta \rangle  = 
\langle  (T_{\gamma} \, ^t \!N) (\zeta,\xi) +  [\, ^t \!N^2 \zeta,\xi]^\gamma -
\, ^t \!N^2 [\zeta,\xi]^\gamma \, , \, Y \rangle ,
\end{equation}
\begin{equation}\label{4}
\langle  (T_{\mu+ \gamma} \N) (\xi, Y) \, , \, Z \rangle 
= 
\langle (T_{\mu}  N) (Y,Z) \, , \, \xi \rangle + 
\d_\mu (\, ^t \!N^2 \xi)(Y,Z) - (\d_\mu \xi)(Y, N^2Z).
\end{equation}

If condition $N^2 = \lambda \, {\mathrm{Id}}_A$ is satisfied,
equations \eqref{1}, \eqref{2}, \eqref{3} and \eqref{4} simplify and we
recover the result of Theorem \ref{A3bialg}.

\begin{remark}\label{Airred}\rm{ From equations
\eqref{01},\eqref{02}, \eqref{1} and \eqref{3},
we see that the conclusion of Theorem \ref{biNijconverse} is valid
when $(A,\mu)$ or $(A^*,\gamma)$ is an irreducible Lie algebroid.}
\end{remark}

\subsection{Deformations of Lie bialgebroids}
When $(A,\mu)$ and $(A^*,\gamma)$ are Lie algebroids, 
if the torsions of $T_\mu N$ and 
$T_\gamma \, ^t\! N$ vanish, $\mu_N= \{N, \mu\}$ and 
$\gamma_{- \, ^t\!N} =\{ N,\gamma\}$
are Lie algebroid structures on $A$ and $A^*$, respectively.
By \eqref{double}, $\{{\widetilde \N}, {\widetilde \mu} + {\widetilde
  \gamma} \} 
= {\{N,\mu\}} + {\{N, \gamma\}}$, therefore `deforming' by $\N$ the Dorfman bracket 
of the double $A \oplus A^*$, equipped with Courant algebroid structure 
$ \Theta = \mu + \gamma$, amounts to considering the `double' of the
pair of Lie algebroids  $(A, \mu_N)$ and $(A^*,\gamma_{- \, ^t\!N})$.
However the Lie algebroids  $(A, \mu_N)$ and $(A^*,\gamma_{- \, ^t\!N})$
do not in general constitute a Lie bialgebroid.

\begin{theorem}\label{bialg}
Let  $((A, \mu), (A^*,\gamma))$ be a Lie bialgebroid.
Assume that $N$ is proportional to a cps structure on $A$ and that 
$ \, ^t\!N$ is  proportional to a cps structure on $A^*$. Then 
$((A, \mu_N),(A^*,\gamma_{-\, ^t\!N}))$  is a
Lie bialgebroid and 
its double is the Courant
algebroid $(A\oplus A^*, \{\widetilde \N,  \mu + \gamma\})$.
\end{theorem}
\begin{proof}
This result is a corollary of Proposition \ref{propdouble} 
and Theorem \ref{biNij}.
\end{proof}

\begin{remark}\rm{
It is possible to consider the deformation of a Lie
bialgebroid by a pair of unrelated 
vector bundle endomorphisms, $N : A \to A$ and $N': A^* \to A^*$, satisfying
$T_\mu N =$ and $T_\gamma N' = 0$. The condition
for the pair of Lie algebroids $(A, \{N,\mu\})$ and $(A^*, \{N', \gamma\})$
to constitute a Lie bialgebroid is
\begin{equation}
\{\{N,\mu\} + \{N',\gamma\} , \{N,\mu\}+ \{N', \gamma\}\} = 0.
\end{equation}
Given that $\{\{N,\mu\}, \{N,\mu\}\} =0$
and  $\{\{N',\gamma\},\{N',\gamma\}\} =0$, this condition becomes
\begin{equation}\label{NN'}
\{\{N,\mu\} , \{N',\gamma\}\} = 0.
\end{equation}
This compatibility condition, $\{\mu_{N},\gamma_{N'}\}=0$, 
means that each deformed structure, $\mu_{N}$
and $\gamma_{N'}$, is a cocycle for the other, or equivalently, that
$\d_{\mu_{N}}$ is a derivation of $[~,~]_{\gamma_{N'}}$, or that
$\d_{\gamma_{N'}}$ is a derivation of $[~,~]_{\mu_{N}}$.
(If $N'= \rm{Id}_{A^*}$, then condition \eqref{NN'}
means that $\d_{\mu_{N}}$ is a derivation of $[~,~]_{\gamma}$, or
$\d_{\gamma}$ is a derivation of
$[~,~]_{\mu_{N}}$. This result is in \cite{nunes}, theorem 3.1.)}
\end{remark}

\subsection{Deformations of trivial Lie bialgebroids}\label{trivial}
We now consider the particular case of the trivial Lie bialgebroids,
such as the generalized tangent bundles. 
It follows from Theorem \ref{A3bialg} that, 
if $((A, \mu), (A^*, 0))$ is the trivial 
Lie bialgebroid associated
with the Lie algebroid $(A,\mu)$, then 
\begin{equation}
\widetilde{{T}_{\mu}\N}
=\widetilde{T_\mu N}.
\end{equation}
In particular, in the case of 
a trivial Lie bialgebroid  $((A, \mu),(A^*,0))$, deforming the
Dorfman bracket of the double by
$\N$ amounts to deforming $(A, \mu)$ by $N$, and Proposition
\ref{propdouble}, Theorem \ref{biNij} and Remark \ref{Airred}
imply the following. 
\begin{corollary}\label{4.4}
Let  $(A, \mu)$ be a Lie algebroid, and let $N$ be a vector bundle
endomorphism of $A$. Let $[~,~]$ be the Dorfman bracket of
the double of the trivial Lie bialgebroid  $((A, \mu),(A^*,0))$, and
let $\N = \begin{pmatrix}N &  0 \\0 & - \, ^t \!N \end{pmatrix}$.

\noindent (i) The deformed bracket $[~,~]_\N$ is the double of the
bracket $[~,~]^\mu_N$.

\noindent (ii)
If $T_\mu N$
vanishes, then $((A, \mu_N),(A^*,0))$ is a  trivial Lie bialgebroid. 
 
\noindent (iii)
If $N$ is proportional to a cps
structure on $A$, then the torsion of $\N$ vanishes.

\noindent (iv)
Conversely, if the torsion of $\N$
vanishes, and if $A$ is irreducible, 
then $N$ is proportional to a cps structure on $A$.

\noindent (v)
If $N$ is proportional to a cps
structure on $A$, then
the double of the trivial Lie bialgebroid 
 $((A, \mu_N),(A^*,0))$  is the Courant
algebroid $(A\oplus A^*, \{\widetilde \N,\mu\})$.

\end{corollary}

For the case of a generalized tangent bundle, 
$TM \oplus T^*M$, parts (i) and (ii) of Corollary \ref{4.4} 
were proved in theorems 2 and 3  of \cite{CGM}.
It was also proved in theorem 3 that, 
when the base manifold $M$ is
connected, 
if the torsions of $N$ and $\N$ both vanish, then $N$ is 
proportional to a cps structure on $TM$.
Since, by lemma 2 of \cite{CGM}, a tangent bundle over a
connected base is an irreducible Lie algebroid, 
this result is implied by (iv) above.

\medskip

There is a more interesting result that does not require  $N^2$ to be
a scalar multiple of the identity.

\begin{theorem}\label{theoremweakdef}
Let $(A, \mu)$ be a Lie algebroid, and let $N$ be a vector\break bundle
endomorphism of $A$. If $N$ is a Nijenhuis tensor for $(A, \mu)$, then 
$\N = \begin{pmatrix}N &  0 \\0 & - \, ^t \!N \end{pmatrix}$ is a weak
deforming tensor  for the Courant
algebroid $(A\oplus A^*, \mu)$, 
and $\{{\widetilde\N}, \mu\}$ is a
Courant algebroid structure on $A\oplus A^*$, which is 
the double of the trivial Lie bialgebroid defined 
by $(A, \mu_N)$ .
\end{theorem}

\begin{proof}
The hypothesis $T_\mu N=0$ is equivalent to $\{\{N,\mu\},N\} = \{\mu,N^2\}$.
Because $\{\mu,\mu\}=0$, this relation 
implies that  $\{\{N,\mu\},N\}$ is a $\d_\mu$-cocycle
and therefore that $\{\{{\widetilde \N}, \mu \},{\widetilde \N}\}$ is
a $\d_\mu$-cocycle. Therefore $\N$ is a weak
deforming tensor for $\mu$.  The Courant algebroid structure
 $\{{\widetilde\N}, \mu\}$ is then $\{N, \mu\}$, i.e., 
the double of the trivial Lie
bialgebroid $((A,\mu_N), (A^*,0))$.
\end{proof}

\subsection{Compatible structures and deforming tensors}\label{compatible}

We shall show that various types of composite structures on Lie algebroids,
for which see, e.g., \cite{yksM} 
\cite{yksR} and references cited there, give rise
to infinitesimal deformations of the Dorfman bracket 
of the double of any trivial Lie
bialgebroid. We assume that 
 $(A, \mu)$ is a Lie algebroid, and we consider 
the trivial Lie bialgebroid  $((A, \mu),(A^*,0))$,

\begin{proposition}\label{PN}
Let $N$ be a vector bundle endomorphism of $A$, and 
let $\pi$ be a
bivector on $A$ such that $N \pi = \pi \, ^t\!N$. 
 If $(\pi,N)$ is a
  PN-structure on $A$, then the skew-symmetric endomorphism of $A
  \oplus A^*$,
$\N = \begin{pmatrix}N &  \pi \\0 & - \, ^t \!N \end{pmatrix}$ is a weak deforming
  tensor for $(A \oplus A^*, \mu)$.
\end{proposition}
\begin{proof}
We denote by $C_\mu(\pi, N) = \{ \pi,\{N , \mu\}\}+\{N,\{ \pi,
\mu\} \}$ the tensor whose vanishing expresses the compatibility
of a Poisson structure $\pi$ and a Nijenhuis tensor $N$ on $A$. We compute
$$
\{\{ \widetilde{\N}, \mu\}, \widetilde{\N}\}
= \{\{N + \pi, \mu\}, N + \pi\}
$$
$$
=
\{\{N , \mu\}, N \} +\{\{ \pi, \mu\},  \pi\}+ \{\{N , \mu\},  \pi\}+\{\{ \pi,
\mu\}, N \}
$$
$$ 
= \{\{N , \mu\}, N \} + [\pi,\pi]^\mu - C_\mu(\pi, N).
$$
Here $[~,~]^\mu$ is the Schouten-Nijenhuis bracket of multivectors. 
Therefore, if we assume that  
$\pi$ is a Poisson bivector and that $N$ and ${\pi}$ are compatible,
then
$\{\{\widetilde{\N}, \mu\},\widetilde{\N}\}=\{\{N , \mu\}, N\}$. 
When $N$ is a Nijenhuis tensor on $(A,\mu)$, 
$\{\{N , \mu\}, N\}$ is a
$\d_\mu$-cocycle
and therefore $\{\{\widetilde{\N}, \mu\},\widetilde{\N}\}$ is a
$\d_\mu$-cocycle.
\end{proof}

As a consequence we recover the well-known fact that when 
$(\pi,N)$ is a
  PN-structure on $A$, then $\{\widetilde{\N}, \mu\} = \{N,\mu\} + \{\pi,
  \mu\}$ is a Courant algebroid structure on $A \oplus A^*$, the
  double of the Lie bialgebroid $((A, \mu_N),(A^*, \gamma_\pi))$, where
    $\gamma_\pi = \{\pi, \mu \}$. See, e.g., theorem 4 of \cite{yksR}.

If $N^2$ is proportional to the identity of $A$ and if $\pi$ is a
bivector such that $N \pi = \pi \, ^t\!N$, then  $\N^2$ is
proportional to the identity of $A \oplus A^*$ and
$\widetilde{{T}_{\mu}(\N)}$ is identified with 
${T_\mu(N)}  - \frac{1}{2}
[\pi,\pi]^\mu +\frac{1}{2} C_\mu(\pi, N)$ in $\A^3$. Using the
bigrading of $\A$, we conclude,

\begin{proposition}\label{above}
If $N$ is proportional to an almost cps structure on $A$ and $\pi$ is a
bivector such that $N \pi = \pi \, ^t\!N$, then
${{T}_{\mu}(\N)} = 0$ if and only if $(\pi, N)$ is a
PN-structure.
\end{proposition}

We can also relate $\Omega$N-structures with deforming
tensor
s, obtaining an analogue of Proposition \ref{PN}, 
although there is no obvious analogue of Proposition \ref{above}.

\begin{proposition}\label{OmegaN} 
 Let $N$ be a vector bundle endomorphism of $A$,
 and let $\omega$ be a $2$-form on $A$ such that $\omega N = \, ^t\!N \omega$.
If $(\omega,N)$ is an $\Omega$N-structure on $A$, then the
skew-symmetric endomorphism of $A \oplus A^*$, $\N= 
\begin{pmatrix}N & 0 \\ \omega & - \, ^t \!N\end{pmatrix}$ is a weak
deforming tensor for $(A \oplus A^*, \mu)$.
\end{proposition}

\begin{proof}
We compute
$$
\{\{\{ \widetilde{\N}, \mu\}, \widetilde{\N}\}
= \{\{N + \omega, \mu\}, N + \omega\}
$$
$$
=
\{\{N , \mu\}, N \} + \{\{N , \mu\},  \omega\}+\{\{ \omega,
\mu\}, N \}
$$
since $\{\{ \omega, \mu\},  \omega\} =0$.
When $(\omega,N)$ is an $\Omega$N-structure, both $\d_\mu\omega =
\{\mu,\omega\}$ and $\d_{\mu_N}\omega =
\{\{N, \mu\},\omega\}$ vanish. 
We conclude, using the vanishing
 of the torsion of $N$, as in the proof of
 Proposition \ref{PN}.
\end{proof}

\medskip

\noindent{\bf{Acknowledgments}} \,
My sincere thanks to Paulo Antumes, Joana Nunes da Costa, Janusz Grabowski,
Camille Laurent-Gengoux, Dmitry Roytenberg and the anonymous referees 
for their comments on a preliminary version of this paper.

\end{document}